\documentclass[12pt, leqno,twoside]{article}
\usepackage{amssymb}
\usepackage{amsmath}
\numberwithin{equation}{section}
\usepackage{amsthm}
\usepackage{graphicx}
\usepackage[pagebackref,colorlinks,citecolor=blue,linkcolor=blue]{hyperref}
\usepackage{cite}
\usepackage{tikz-cd}
\usepackage{todonotes}
\usepackage{amsmath}
\usepackage{amsfonts}
\usepackage{amssymb}
\usepackage{amsmath,bbm,amssymb,amsxtra}
\usepackage{mathrsfs}
\usepackage{enumerate}
\usepackage{caption}
\allowdisplaybreaks
\usepackage{epsfig}
\usepackage{ae}
\usepackage{epstopdf}
\def\vint{\mathop{\mathchoice%
         {\setbox0\hbox{$\displaystyle\intop$}\kern 0.22\wd0%
          \vcenter{\hrule width 0.6\wd0}\kern -0.82\wd0}%
         {\setbox0\hbox{$\textstyle\intop$}\kern 0.2\wd0%
          \vcenter{\hrule width 0.6\wd0}\kern -0.8\wd0}%
         {\setbox0\hbox{$\scriptstyle\intop$}\kern 0.2\wd0%
          \vcenter{\hrule width 0.6\wd0}\kern -0.8\wd0}%
         {\setbox0\hbox{$\scriptscriptstyle\intop$}\kern 0.2\wd0%
          \vcenter{\hrule width 0.6\wd0}\kern -0.8\wd0}}%
         \mathopen{}\int}

\newcommand{\diam}{\,\text{\rm diam\,}}

\newcommand{\real}{{\mathbb R}}

\newtheorem{thm}{Theorem}[section]

\newtheorem{cor}[thm]{Corollary}
\newtheorem{lem}[thm]{Lemma}
\newtheorem{prop}[thm]{Proposition}

\newtheorem{claim}{Claim}[section]
\newtheorem{subclaim}{Subclaim}
\newtheorem{conj}[equation]{Conjecture}
\newtheorem{case}{Case}[section]
\newtheorem*{mysolution}{Solution}
\newtheorem{step}{Step}[section]
\theoremstyle{definition}
\newtheorem{defn}[thm]{Definition}
\newtheorem{example}[thm]{Example}
\newtheorem{prob}[equation]{Problem}
\newtheorem{ques}[equation]{Question}
\newtheorem{rem}[thm]{Remark}
\newtheorem{rems}{Remarks}[section]
\newcounter {own}
\def\theown {\thesection       .\arabic{own}}

\newenvironment{pf}[1][]{%
 \vskip 3mm
 \noindent
 \ifthenelse{\equal{#1}{}}%
  {{\slshape Proof. }}%
  {{\slshape #1.} }%
 }%
{\qed\bigskip}

\newcounter{alphabet}

\newenvironment{Thm}[1][]{\refstepcounter{alphabet}%
\bigskip%
\noindent%
{\bf Theorem \Alph{alphabet}}%
\ifthenelse{\equal{#1}{}}{}{ (#1)}%
{\bf .} \itshape}{\vskip 8pt}

\makeatletter

\newcommand{\id}{{\operatorname{id}}}

\newcommand{\RNum}[1]{\uppercase\expandafter{\romannumeral #1\relax}}

\def\be{\begin{equation}}
\def\ee{\end{equation}}

\newcommand{\ben}{\begin{enumerate}}
\newcommand{\een}{\end{enumerate}}

\newcommand{\blem}{\begin{lem}}
\newcommand{\elem}{\end{lem}}
\newcommand{\bthm}{\begin{thm}}
\newcommand{\ethm}{\end{thm}}
\newcommand{\bcor}{\begin{cor}}
\newcommand{\ecor}{\end{cor}}
\newcommand{\beg}{\begin{examp}}
\newcommand{\eeg}{\end{examp}}
\newcommand{\begs}{\begin{examples}}
\newcommand{\eegs}{\end{examples}}
\newcommand{\bdefe}{\begin{defn}}
\newcommand{\edefe}{\end{defn}}
\newcommand{\bprob}{\begin{prob}}
\newcommand{\eprob}{\end{prob}}
\newcommand{\bques}{\begin{ques}}
\newcommand{\eques}{\end{ques}}
\newcommand{\bei}{\begin{itemize}}
\newcommand{\eei}{\end{itemize}}
\newcommand{\bcl}{\begin{claim}}
\newcommand{\ecl}{\end{claim}}
\newcommand{\bscl}{\begin{subclaim}}
\newcommand{\escl}{\end{subclaim}}
\newcommand{\bca}{\begin{case}}
\newcommand{\eca}{\end{case}}
\newcommand{\bstep}{\begin{step}}
\newcommand{\estep}{\end{step}}
\newcommand{\bsol}{\begin{mysolution}}
\newcommand{\esol}{\end{mysolution}}
\newcommand{\bcon}{\begin{conj}}
\newcommand{\econ}{\end{conj}}
\newcommand{\bcons}{\begin{conjs}}
\newcommand{\econs}{\end{conjs}}
\newcommand{\bprop}{\begin{prop}}
\newcommand{\eprop}{\end{prop}}
\newcommand{\br}{\begin{rem}}
\newcommand{\er}{\end{rem}}
\newcommand{\brs}{\begin{rems}}
\newcommand{\ers}{\end{rems}}
\newcommand{\bo}{\begin{obser}}
\newcommand{\eo}{\end{obser}}
\newcommand{\bos}{\begin{obsers}}
\newcommand{\eos}{\end{obsers}}

\newcommand{\bpf}{\begin{pf}}
\newcommand{\epf}{\end{pf}}

\newcommand{\ba}{\begin{array}}
\newcommand{\ea}{\end{array}}
\newcommand{\beq}{\begin{eqnarray}}
\newcommand{\beqq}{\begin{eqnarray*}}
\newcommand{\eeq}{\end{eqnarray}}
\newcommand{\eeqq}{\end{eqnarray*}}

\begin{document}
\title{\Large\bf Locally biH\"{o}lder continuous mappings and their induced embeddings between Besov spaces
 \footnotetext{\hspace{-0.35cm}
  $2020$ {\it Mathematics Subject classfication}: 30L10, 46E36. 
 \endgraf{{\it Key words and phases}: Locally biH\"{o}lder continuous mapping, (power) quasisymmetric mapping, Besov space, Ahlfors regular metric space, embedding, uniform boundedness.}
}
}
\author{Manzi Huang, Xiantao Wang, Zhuang Wang, and Zhihao Xu}
\date{ }
\maketitle

\begin{abstract}
	In this paper, we introduce a class of homeomorphisms between metric spaces, which are locally biH\"{o}lder continuous mappings. Then an embedding result between Besov spaces induced by locally  biH\"{o}lder continuous mappings between Ahlfors regular spaces is established, which extends the corresponding result of  Bj\"{o}rn-Bj\"{o}rn-Gill-Shanmugalingam (J. Reine Angew. Math. 725: 63-114, 2017). Furthermore, an example is constructed to show that our embedding result is more general. We also introduce a geometric condition, named as uniform boundedness, to characterize when a quasisymmetric mapping between uniformly perfect spaces is locally biH\"{o}lder continuous.
\end{abstract}

 \section{Introduction}
 For nearly three decades, the analysis on metric measure spaces has been under active study, e.g., \cite{BB11,BBS,BBS03,H03,HP,H,HK,HKST}. Given a metric measure space $(Z, d_Z, \nu_Z)$, many function spaces defined on this space have been well established, e.g., Sobolev spaces, Besov spaces and Triebel-Lizorkin spaces (see \cite{N00,H96,GKS10,BP03,KYZ,KYZ10,GKZ13} and the references therein).

 Given a homeomorphism $f$ between metric spaces $(Z,d_Z)$ and $(W,d_W)$, one natural question is that what kind of correspondence between certain function spaces on $(Z,d_Z, \nu_Z)$ and $(W,d_W, \nu_W)$ can be induced by $f$.
  The question has been extensively studied for many function spaces when $f$ is a quasiconformal or quasisymmetric mapping, including  Sobolev spaces, Besov spaces, Triebel-Lizorkin spaces and other related function spaces (see \cite{BBGS,BoSi, HKM92,Rie,KYZ,HenK,Vo,KXZZ} and the references therein).

 In a very recent work by Bj\"{o}rn-Bj\"{o}rn-Gill-Shanmugalingam \cite{BBGS}, the question was studied when the homeomorphism $f$ is a biH\"{o}lder continuous mapping and the underlying spaces are bounded Ahlfors regular spaces $(Z,d_Z, \nu_Z)$ and $(W,d_W, \nu_W)$. It was shown that $f$ induces an embedding between Besov spaces $B^{s}_{p, p}(W)$ and $B^{ s^\prime}_{p, p}(Z)$ for suitable $s$ and $s^\prime$, via composition; see \cite[Proposition 7.2]{BBGS} for the details. Recall that for $\theta_1>0$ and $\theta_2>0$, a homeomorphism $f: (Z, d_Z)\to (W, d_W)$ is called $(\theta_1, \theta_2)$-{\it biH\"{o}lder continuous} if there exists a constant $C\geq1$ such that  for all $x, y\in Z$,
 \begin{equation}\label{defn-biholder-intro}
 	C^{-1}d_Z(x, y)^{\theta_1}\leq d_W(f(x), f(y))\leq Cd_Z(x, y)^{\theta_2}.
 \end{equation}
 In particular, if $\theta_1=\theta_2$, then  $f$ is called a {\it snowflake mapping}.

It is interesting to ask what can remain of the conclusion of \cite[Proposition 7.2]{BBGS} if the assumption that the underlying metric spaces $(Z, d_Z)$ and $(W, d_W)$ are bounded is removed.
	As the first purpose of this paper, we consider this question.
	However, the assumption of boundedness on the underlying spaces plays a key role, since for a biH\"{o}lder continuous mapping $f: (Z, d_Z)\to (W, d_W)$, if $(Z, d_Z)$ is Ahlfors regular and unbounded, then $f$ must be a snowflake mapping. To avoid such a constraint, let us introduce the following class of mappings.
Before the statement of the definition, we make the following conventions: $(1)$ For a subset $A$ of $(Z, d_Z)$, we use $\diam A$ to denote the diameter of $A$, that is, $\diam A=\sup\{d_Z(z_1, z_2): z_1, z_2\in A\}$. $(2)$ All metric spaces involved in this paper are assumed to contain at least two points. $(3)$ When $(Z, d_Z)$ is unbounded, we take $\diam Z=\infty$. Then for any metric space $(Z, d_Z)$, $0<\diam Z\leq \infty$.

  \begin{defn}\label{LbHC}
 	Let $\theta_1>0$, $\theta_2>0$ and $0<r<2 \diam Z$.  A homeomorphism $f: (Z, d_Z)\to (W, d_W)$ is called {\it locally $(\theta_1, \theta_2, r)$-biH\"{o}lder continuous} if every pair of points $x, y\in Z$ satisfies the condition \eqref{defn-biholder-intro} provided that $d_Z(x, y)<r$. Also, the constant $C$ in \eqref{defn-biholder-intro} is called a {\it locally biH\"{o}lder continuity coefficient} of $f$.
 \end{defn}

Obviously, every biH\"{o}lder continuous mapping is locally biH\"{o}lder continuous, while the converse is not true (See Example \ref{ex} below). The following are direct consequences of the definitions.

\bprop\label{1-8-8}
$(1)$ If $f$ is $(\theta_1, \theta_2)$-biH\"{o}lder continuous with a biH\"{o}lder continuity coefficient $C_1$, then the inverse $f^{-1}$ of $f$ is $(1/\theta_2, 1/\theta_1)$-biH\"{o}lder continuous with a biH\"{o}lder continuity coefficient $C_2=\max\{C_1^{1/ \theta_1},\; C_1^{1/ \theta_2}\}$.

$(2)$ If $f$ is locally $(\theta_3, \theta_4, r)$-biH\"{o}lder continuous with a locally biH\"{o}lder continuity coefficient $C_3$, then the inverse $f^{-1}$ of $f$ is locally $(1/\theta_4, 1/\theta_3, C_3^{-1}r^{\theta_3})$-biH\"{o}lder continuous with a locally biH\"{o}lder continuity coefficient $C_4=\max\{C_3^{1/ \theta_3},\; C_3^{1/ \theta_4}\}$.
\eprop

The following result is our answer to the aforementioned question, which provides us with embeddings between Besov spaces induced by locally biH\"{o}lder continuous mappings.

\begin{thm}\label{thm-1}
	Assume that $(Z, d_Z, \nu_Z)$ and $(W, d_W, \nu_W)$ are Ahlfors $Q_Z$-regular and Ahlfors $Q_W$-regular spaces with $Q_Z>0$ and $Q_W>0$, respectively, and let $\theta_1>0$, $\theta_2>0$, $s>0$, $s^\prime>0$ and $p\geq 1$
be constants such that
	\begin{equation}\label{s-s-relation}
	Q_Z\geq \theta_1 Q_W\;\;\mbox{and}\;\;	s^\prime\leq \theta_2 s+\frac{\theta_2 Q_W-Q_Z}{p}.
	\end{equation}
 Suppose that $f: (Z, d_Z)\rightarrow (W, d_W)$ is a locally $(\theta_1, \theta_2, r)$-biH\"{o}lder continuous mapping with $0<r<2\,\diam Z$. Then $f$ induces a canonical bounded  embedding $f_{\#}: B^{s}_{p, p}(W)\rightarrow B^{ s^\prime}_{p, p}(Z)$ via composition.
\end{thm}

The terminology appeared in Theorem \ref{thm-1} and in the rest of this section will be introduced in Section \ref{sec-2} unless stated otherwise.

\br
Theorem \ref{thm-1} is a generalization of \cite[Proposition 7.2]{BBGS}. This is because when $(Z, d_Z)$ is bounded and $r>\diam Z$, Theorem \ref{thm-1} coincides with \cite[Proposition 7.2]{BBGS}. Further, Example \ref{ex} and Remark \ref{sec5}$(ii)$ below show that Theorem \ref{thm-1} is more general than \cite[Proposition 7.2]{BBGS}.
\er

As we know, a quasisymmetric mapping between bounded uniformly perfect spaces is locally biH\"{o}lder continuous since it follows from \cite[Theorem 3.14]{TV} or \cite[Corollary 11.5]{H} that it is biH\"{o}lder continuous.
 Naturally, one will ask if there is any analog for the case when the underlying spaces are unbounded. However, Example \ref{ex-add} below shows that not every quasisymmetric mapping between unbounded uniformly perfect spaces is locally biH\"{o}lder continuous.
As the second purpose of this paper, we seek for a characterization for a quasisymmetric mapping to be locally biH\"{o}lder continuous.
Before the statement of our result, let us introduce the following concept.

 \begin{defn}
 	For $0<r< 2 \diam Z$, a homeomorphism $f: (Z, d_Z)\to (W, d_W)$ is called $r$-{\it uniformly bounded} if there exist constants $a$ and $b$ with $0<a<b$ such that for all $x\in Z$,
 	$$a \leq \diam f\big(B(x, r)\big)\leq b,$$ where $B(x, r)=\{y\in Z:\; d_Z(y, x)<r\}$, i.e., the open ball in $Z$ with center $x$ and radius $r$.
  \end{defn}

The following property shows that in the definition of $r$-uniform boundedness, the exact value of the parameter $r$ is not important for quasisymmetric mappings.

\begin{prop}\label{uniformly-bounded}
Suppose that $(Z, d_Z)$ is a $\kappa$-uniformly perfect space with $\kappa>1$ and $f: (Z, d_Z)\to (W, d_W)$ is $\eta$-quasisymmetric. If $f$ is $r$-uniformly bounded for an $r$ with $0<r< 2 \diam Z$, then $f$ is $s$-uniformly bounded for any $s\in (0,2\diam Z)$.
\end{prop}

Note that quasisymmetry in a uniformly perfect space implies power quasisymmetry (See Theorem $A$ below).
Based on the uniform boundedness, we obtain the following geometric characterization for a (power) quasisymmetric mapping between unbounded uniformly perfect spaces to be locally biH\"{o}lder continuous.

\begin{thm}\label{thm-2}
Suppose that $(Z, d_Z)$ is $\kappa$-uniformly perfect with $\kappa>1$, and $f: (Z, d_Z)\to (W, d_W)$ is a $(\theta, \lambda)$-power quasisymmetric mapping with $\theta\geq 1$ and $\lambda\geq 1$. Then for any $r\in (0,2 \diam Z)$, the following are quantitatively equivalent: \ben
\item[$(1)$]
$f$ is locally $(\theta, 1/\theta, r)$-biH\"{o}lder continuous.
\item[$(2)$]
$f$ is $r$-uniformly bounded.
\een
\end{thm}
Here, for two conditions, we say that Condition $\Phi$ quantitatively implies Condition $\Psi$ if Condition $\Phi$ implies Condition $\Psi$ and the data of Condition $\Psi$ depends only on that of Condition $\Phi$. If Condition $\Psi$ also quantitatively implies Condition $\Phi$, then we say that Condition $\Phi$ is equivalent to Condition $\Psi$, quantitatively.

\br
In Theorem \ref{thm-2}, the assumption that $(Z, d_Z)$ is uniformly perfect cannot be removed. For example, the identity mapping of integers $\id: \mathbb Z\rightarrow \mathbb Z$ with the standard Euclidean distance is $1$-biLipschitz, and thus, it is power quasisymmetric, and locally $(\theta_1,\theta_2,r)$-biH\"{o}lder continuous for any $\theta_1>0$, $\theta_2>0$ and $0<r<1$.
 However, it is not $s$-uniformly bounded for any $s\in (0,1)$.
 \er

Throughout this paper, the letter $C$ (sometimes with a subscript) denotes a positive constant that depends only on the given parameters of the spaces and may change at different occurrences. The notation $A\lesssim B$ (resp. $A \gtrsim B$) means that there is a constant $C_1\geq 1$ (resp. $C_2\geq 1$) such that $A \leq C_1 \cdot B$ (resp. $A \geq C_2 \cdot B).$ We also call $C_1$ and $C_2$ comparison coefficients of $A$ and $B$. In particular, $C_1$ (resp. $C_2$) is called an upper comparison coefficient (resp. a lower comparison coefficient) for $A$ and $B$. If $A\lesssim B$ and $A \gtrsim B$, then we write $A\approx B$.

The paper is organized as follows. In Section \ref{sec-2}, some basic concepts and known results will be introduced. Section \ref{sec-3} will be devoted to the proof of Theorem \ref{thm-1}. In Section \ref{sec-4}, the proofs of Proposition \ref{uniformly-bounded} and Theorem \ref{thm-2}
will be presented, and in Section \ref{sec-5}, two examples will be constructed.

\section{Basic terminologies}\label{sec-2}
In this section, we introduce some necessary notions and notations.

A metric space $(Z, d_Z)$ is called $\kappa$-{\it uniformly perfect} with $\kappa> 1$ if for each $x\in Z$ and for each $r>0$, the set $B(x, r)\setminus B(x, r/\kappa)$ is nonempty whenever the set $Z\setminus B(x, r)$ is nonempty. Sometimes, $(Z, d_Z)$ is called {\it uniformly perfect} if $Z$ is $\kappa$-{\it uniformly perfect} for some $\kappa > 1$.

\blem\label{1-8-4}
Suppose that $(Z, d_Z)$ is $\kappa$-uniformly perfect with $\kappa>1$, and let $x\in Z$. Then for any $r\in(0,2\diam Z)$, there exists $z\in Z$ such that $$\frac{r}{\mu}\leq d_Z(x, z)<r,$$ where $\mu=\max\{8, \kappa\}$.
\elem
\bpf Let $x\in Z$.
Since $0<r<2\diam Z$, we see that there exists $y\in Z$ such that $d_Z(x, y)>r/8$. If $d_Z(y, x)< r$, by letting $z=y$, we see that the lemma is true. If $d_Z(y, x)\geq r$, then the uniform perfectness of $(Z,d_Z)$ implies that there is $y^\prime\in Z$ such that $r/\kappa\leq d_Z(x, y^\prime)<r$. By letting $z=y^\prime$, we know that the lemma holds true as well.
\epf

A homeomorphism $f: (Z, d_Z)\to (W, d_W)$ is called {\it $\eta$-quasisymmetric} if there exists a self-homeomorphism $\eta$ of $[0, +\infty)$ such that for all triples of points $x, y, z\in Z$,
	\begin{equation}\label{eta}
	\frac{d_W(f(x), f(z))}{d_W(f(y), f(z))}\leq \eta\left(\frac{d_Z(x, z)}{d_Z(y, z)}\right).
	\end{equation}
	In particular, if there are constants $\theta\geq 1$ and $\lambda\geq 1$ such that
	\begin{equation*}\label{eq-1.1}
		\eta_{\lambda, \theta}(t)=
		\left\{\begin{array}{cl}
			\lambda t^{\frac{1}{\theta}}& \text{for} \;\; 0<t<1, \\
			\lambda t^{\theta}& \text{for} \;\; t\geq 1,
		\end{array}\right.
	\end{equation*}
	then $f$ is called a $(\theta, \lambda)$-{\it power quasisymmetric mapping}. Here, the notation $\eta_{\lambda, \theta}$ means that the control function $\eta$ depends only on the given parameters $\theta$ and $\lambda$.
	
\begin{Thm}[{\cite[Theorem 11.3]{H}}]\label{Thm-A}
An $\eta$-quasisymmetric mapping of a uniformly perfect space is $(\theta, \lambda)$-power quasisymmetric, quantitatively.
\end{Thm}
	
In the following, we always use the notation $(Z, d_Z, \nu_Z)$ to denote a metric space $(Z, d_Z)$ admitting a Borel regular measure $\nu_Z$. A metric measure space $(Z, d_Z, \nu_Z)$ is called
 \ben
 \item[$(1)$]
 {\it doubling} if there exists a constant $C\geq 1$ such that for all $x\in Z$ and $0<r<2\diam Z$,
$$0< \nu_Z(B(x, 2r))\leq C \nu_Z\big(B(x, r)\big)<\infty.$$

\item[$(2)$]
 {\it $Q_Z$-Ahlfors regular} with $Q_Z>0$ if there exists a constant $C\geq 1$ such that for all $z\in Z$ and $0<r<2 \diam Z$,
$$
C^{-1} r^{Q_Z}\leq \nu_Z\Big(B(z, r)\big) \leq C r^{Q_Z}.
$$
\een

It is known that every Ahlfors regular space is doubling and uniformly perfect   (cf. \cite[Section 11]{H}).

For  given $1\leq p<\infty$, $s>0$  and a function $u: Z\to\mathbb R$, the {\it homogeneous Besov norm} on the metric measure space $(Z, d_Z, \nu_Z)$ is defined by
\begin{equation}\label{eq-besov}
	\|u\|_{\dot B_{p, p}^{s}(Z)}=\left(\int_Z \int_Z\frac{|u(x)-u(y)|^p}{d_Z(x, y)^{sp}}\frac{d\nu_Z(x)\nu_Z(y)}{\nu_Z(B(x, d_Z(x, y)))}\right)^{1/p}.
\end{equation}

We write the {\it homogeneous Besov space} $\dot B_{p, p}^{s}(Z)$ for the subspace of $L^p_{\text{loc}}(Z)$ consisting of all functions $u$ such that $$\|u\|_{\dot B_{p, p}^{s}(Z)}<\infty.$$

 We note that, properly speaking, \eqref{eq-besov} is actually a seminorm on $L^p_{\text{loc}}(Z)$ since any constant function has Besov norm 0.
We define the {\it Besov space} $B_{p, p}^{s}(Z)$ to be the normed space of all measurable functions $u\in L^{p}(Z)$ such that
$$
\|u\|_{B_{p, p}^{s}(Z)}=\|u\|_{L^{p}(Z)}+\|u\|_{\dot B_{p,p}^{s}(Z)}<\infty.
$$

\section{Locally biH\"{o}lder continuous mappings and their induced embeddings}\label{sec-3}

The aim of this section is to prove Theorem \ref{thm-1}. Before the proof, we need some preparation which consists of the following two auxiliary lemmas.

\begin{lem}\label{lem-besov-norm}
	Suppose that $(Z, d_Z, \nu_Z)$ is Ahlfors $Q_Z$-regular with $Q_Z>0$.  Let $n_0\in \mathbb Z$, and for each $n\in \mathbb Z$, let $t_n=C\sigma^n$, where $C>0$ and $0<\sigma<1$.
For any $s>0$ and $p\geq 1$, if $u\in L^{p}(Z)$, then
	$$\|u\|^p_{B^s_{p, p}(Z)}\approx \|u\|^p_{L^p(Z)}+\sum_{n=n_0}^{+\infty} t_n^{-sp}\int_Z\vint_{B(x, t_n)}|u(x)-u(y)|^p d\nu_Z(y)d\nu_Z(x),$$
where the comparison coefficients depend on $n_0$.
\end{lem}

\begin{proof}
The following estimate easily follows from similar arguments as in the proof of \cite[Theorem 5.2]{GKS10} or \cite[Lemma 5.4]{BBGS}:
\beq\label{1-3-1}
\|u\|^p_{\dot B^s_{p, p}(Z)} \approx  \sum_{n=-\infty}^{+\infty} t_n^{-sp}\int_Z\vint_{B(x, t_n)}|u(x)-u(y)|^p d\nu_Z(y)d\nu_Z(x).
\eeq

Let $n_0\in \mathbb Z$. Then the estimate \eqref{1-3-1} shows that to prove the estimate in the lemma, it suffices to show that
	\begin{equation*}\label{norm-equiv}
		\sum_{n=-\infty}^{n_0} t_n^{-sp}\int_Z\vint_{B(x, t_n)}|u(x)-u(y)|^p d\nu_Z(y)d\nu_Z(x)\lesssim \|u\|^p_{L^p(Z)}.
	\end{equation*}

	Note that
	\begin{align*}
		\int_Z\vint_{B(x, t_n)}|u(x)-u(y)|^p d\nu_Z(y)d\nu_Z(x)&\lesssim \int_Z\vint_{B(x, t_n)} (|u(x)|^p+|u(y)|^p)d\nu_Z(y)d\nu_Z(x)\\
			&=\|u\|^p_{L^p(Z)}+\int_Z\vint_{B(x, t_n)} |u(y)|^pd\nu_Z(y)d\nu_Z(x).
	\end{align*}
	Since $(Z, d_Z, \nu_Z)$ is Ahlfors $Q_Z$-regular, we know that for any $y\in B(x, t_n)$, $$\nu_Z(B(x, t_n))\approx \nu_Z(B(y, t_n)).$$ It follows from the Fubini theorem that
	\begin{align*}
		\int_Z\vint_{B(x, t_n)} |u(y)|^pd\nu_Z(y)d\nu_Z(x)&\approx  \int_Z\int_Z \frac{|u(y)|^p\chi_{B(x, t_n)}(y)}{\nu_Z(B(y, t_n))} d\nu_Z(y)d\nu_Z(x)\\
		&=\int_Z |u(y)|^p d\nu_Z(y) \int_{Z}\frac{\chi_{B(x, t_n)}(y)}{\nu_Z(B(y, t_n))} d\nu_Z(x)\\
		&=\|u\|^p_{L^p(Z)}.
	\end{align*}
	Therefore,
	$$\sum_{n=-\infty}^{n_0} t_n^{-sp}\int_Z\vint_{B(x, t_n)}|u(x)-u(y)|^p d\nu_Z(y)d\nu_Z(x)\lesssim \sum_{n=-\infty}^{n_0} t_n^{-sp} \|u\|^p_{L^p(Z)}\lesssim \|u\|^p_{L^p(Z)},$$
	which is what we need, and hence, the lemma is proved.
\end{proof}


\begin{lem}\label{embedding-lp}
	Assume that $(Z, d_Z, \nu_Z)$ and $(W, d_W, \nu_W)$ are   Ahlfors $Q_Z$-regular and Ahlfors $Q_W$-regular spaces with $Q_Z>0$ and $Q_W>0$, respectively. Let $\theta_1>0$, $\theta_2>0$ and $0<r<2\diam Z$. Suppose that $f: Z\rightarrow W$ is a locally $(\theta_1, \theta_2, r)$-biH\"{o}lder continuous mapping  such that $Q_Z\geq \theta_1 Q_W$. Then for any $p\geq 1$, the mapping $f$ induces a bounded embedding $f_{\#}: L^p(W)\rightarrow L^p(Z)$ via composition.
\end{lem}

When $Z$ and $W$ are bounded and $f$ is biH\"{o}lder continuous, Lemma \ref{embedding-lp} coincides with \cite[Lemma 7.1]{BBGS}. The proof method of \cite[Lemma 7.1]{BBGS} is also applicable to Lemma \ref{embedding-lp}, and so, we omit the details here.

Now, we are ready to prove Theorem \ref{thm-1}.

\subsection*{Proof of Theorem \ref{thm-1}}
Assume that $(Z, d_Z, \nu_Z)$ and $(W, d_W, \nu_W)$ are Ahlfors $Q_Z$-regular and Ahlfors $Q_W$-regular spaces with $Q_Z>0$ and $Q_W>0$, respectively.  Suppose that $f: (Z, d_Z)\rightarrow (W, d_W)$ is a locally $(\theta_1, \theta_2, r)$-biH\"{o}lder continuous mapping with $\theta_1\geq \theta_2>0$ and $r\in (0, 2\diam Z)$.
	
Let $u\in B^{s}_{p, p}(W)$ with $s>0$ and $p\geq 1$, and let $v=u\circ f$.
	Since $u\in L^p(W)$, it follows from Lemma \ref{embedding-lp} that
	\begin{equation}\label{eq-1-4}
	\|v\|_{L^p(Z)}\lesssim \|u\|_{L^p(W)},
	\end{equation}
which implies $v\in L^p(Z)$.

	For each $n\in \mathbb{Z}$, let $t_n=C\sigma^n$, where $C>0$ and $0<\sigma<1$. Apparently, there is $n_0\in \mathbb Z$ such that for $n\in \mathbb{Z}$, if $n\geq n_0$, then $$t_{n}<r.$$  Also, it follows from Lemma \ref{lem-besov-norm} that
	for any $s'>0$,
\begin{align}
		\|v\|^p_{B^{s^\prime}_{p, p}(Z)} \approx \|v\|^p_{L^p(Z)}+\sum_{n=n_0}^{+\infty}  I_n,	\label{zw-10-16}
	\end{align}
 where
$$I_n=t_n^{- s^\prime p}\int_Z\vint_{B(x, t_n)}|v(x)-v(y)|^p d\nu_Z(y)d\nu_Z(x)\notag.$$

In the following, we are going to estimate $I_n$. For this, we first estimate the integral
	$$i_n=\vint_{B(x, t_n)}|v(x)-v(y)|^p d\nu_Z(y).$$

Since for any $n\geq n_0$, $t_n<r$, we infer from the locally biH\"{o}lder continuity of $f$ that there is $C_0\geq 1$ such that $$f(B(x, t_n))\subset B(f(x), C_0t_n^{\theta_2}).$$ Then the Ahlfors regularity of $(Z, d_Z, \nu_Z)$ gives
	\begin{align}\label{1-5-2}
	i_n &\approx \frac{1}{t_n^{Q_Z}} \int_{Z}|v(x)-u\circ f(y)|^p \chi_{B(x, t_n)}(y) d\nu_Z(y)
\\ \nonumber
		&\leq \frac{1}{t_n^{Q_Z}} \int_{Z}|v(x)-u\circ f(y)|^p \chi_{B(f(x), C_0t_n^{\theta_2})}(f(y)) d\nu_Z(y).
	\end{align}

As $v\in L^p(Z)$, we know that $v(x)$ is finite for almost every $x\in Z$.
 Since
	\begin{align*}
 \int_W |v(x)-u(y^\prime)|^p\chi_{B(f(x), C_0t_n^{\theta_2})}(y^\prime)d\nu_W(y^\prime)\lesssim \;& |v(x)|^p\nu_W(B(f(x), C_0t_n^{\theta_2}))\\ &+ \int_{B(f(x), C_0t_n^{\theta_2})}|u(y^\prime)|^pd\nu_W(y^\prime),
 	\end{align*}
 we see from the Ahlfors regularity of $(W, d_W, \nu_W)$ that for almost every $x\in Z$, as a function of $y^\prime$, $|v(x)-u(y^\prime)|^p\chi_{B(f(x), C_0t_n^{\theta_2})}(y^\prime)$ belongs to $L^1(W)$.
  It follows from Lemma \ref{embedding-lp} that for almost every $x\in Z$,
$$\int_{Z}|v(x)-u\circ f(y)|^p \chi_{B(f(x), C_0t_n^{\theta_2})}(f(y)) d\nu_Z(y)\lesssim  \int_W |v(x)-u(y^\prime)|^p\chi_{B(f(x), C_0t_n^{\theta_2})}(y^\prime) d\nu_W(y^\prime),$$
and thus, we deduce from \eqref{1-5-2} that
	\begin{align*}
		i_n \lesssim  t_n^{\theta_2 Q_W-Q_Z} \vint_{B(f(x), C_0t_n^{\theta_2})} |v(x)-u(y^\prime)|^p  d\nu_W(y^\prime).
	\end{align*}This is what we want.

Since $u\in B^{s}_{p, p}(W)$, it follows from Lemma \ref{lem-besov-norm} that for any $n\geq n_0$,
\begin{equation*}
\int_W \vint_{B(x^\prime, C_0t_n^{\theta_2})} |u(x^\prime)-u(y^\prime)|^p  d\nu_W(y^\prime)d\nu_W(x^\prime)<\infty,
\end{equation*}
which shows that $\vint_{B(x^\prime, C_0t_n^{\theta_2})} |u(x^\prime)-u(y^\prime)|^p  d\nu_W(y^\prime)$ belongs to $L^1(W)$ as a function of $x^\prime$.
 Again, by Lemma \ref{embedding-lp}, we obtain that
	\begin{align*}
		I_n&\lesssim t_n^{-s^\prime p+\theta_2 Q_W-Q_Z} \int_Z \vint_{B(f(x), C_0t_n^{\theta_2})} |u\circ f(x)-u(y^\prime)|^p  d\nu_W(y^\prime) d\nu_Z(x)\\
		&\lesssim t_n^{-s^\prime p+\theta_2 Q_W-Q_Z}  \int_W \vint_{B(x^\prime, C_0t_n^{\theta_2})} |u(x^\prime)-u(y^\prime)|^p d\nu_W(y^\prime)d\nu_W(x^\prime).
	\end{align*}
	
Assume that $s$ and $s^\prime$ satisfy the relation \eqref{s-s-relation}. Then we know that for any $n\geq n_0$,
	$$t_n^{-s^\prime p+\theta_2 Q_W-Q_Z}\leq r^{\theta_2(Q_W+sp)-s^\prime p-Q_Z}(t_n^{\theta_2})^{-sp}.$$
	This implies that
	$$I_n\lesssim (t_n^{\theta_2})^{-sp}  \int_W \vint_{B(x^\prime, C_0t_n^{\theta_2})} |u(x^\prime)-u(y^\prime)|^p d\nu_W(y^\prime)d\nu_W(x^\prime).$$

By substituting the estimate of $I_n$ into \eqref{zw-10-16}, we conclude from Lemma \ref{lem-besov-norm} that
	$$\|v\|^p_{B^{s^\prime}_{p, p}(Z)} \lesssim \|u\|^p_{B^{s}_{p, p}(W)}. $$
	
Let $$f_{\#}(u)=u\circ f.$$ Then we have proved that $f_{\#}: B^{s}_{p, p}(W)\rightarrow B^{ s^\prime}_{p, p}(Z)$ is a bounded embedding.
\qed

\section{Power quasisymmetry, locally biH\"{o}lder continuity and uniform boundedness}\label{sec-4}

The purpose of this section is to prove Proposition \ref{uniformly-bounded} and Theorem \ref{thm-2}.

\begin{proof}[{\bf Proof of Proposition \ref{uniformly-bounded}}]

It follows from the assumption of $f$ being $r$-uniformly bounded that there must exist constants $a>0$ and $b>0$ such that for all $x\in Z$,
	 \beq\label{1-7-1}a\leq \diam f\big(B(x, r)\big)\leq b.\eeq

Let $x\in Z$ and $s\in (0, 2\diam Z)$. To prove that $f$ is $s$-uniformly bounded, we only need to consider two cases: $s>r$ and $s<r$. For the first case, it follows from the fact $B(x, r)\subset B(x, s)$ and \eqref{1-7-1} that
 \beq\label{1-7-2}\diam f(B(x, s))\geq \diam f\big(B(x, r)\big)\geq a.\eeq

If $B(x, s)\setminus B(x, r)=\emptyset$, obviously, we obtain from \eqref{1-7-1} that
	\beq\label{1-7-3}
	\diam f(B(x, s))=\diam f\big(B(x, r)\big)\leq b,
	\eeq
and if $B(x, s)\setminus B(x, r)\not=\emptyset$, it follows from the uniform perfectness of $(Z, d_Z)$ that there exists $z\in B(x, r)$ such that $$\frac{r}{\kappa}\leq d_Z(x, z)<r.$$ This indicates that for any $y\in B(x, s)$,
$$\frac{d_Z(x, y)}{d_Z(x, z)}\leq \frac{\kappa s}{r}.$$
Then the $\eta$-quasisymmetry of $f$ gives
	$$\frac{d_W(f(x), f(y))}{d_W(f(x), f(z))}\leq \eta\left(\frac{\kappa s}{r}\right),$$
	and thus, we get
	$$d_W(f(x), f(y))\leq \eta\left(\frac{\kappa s}{r}\right) \diam f\big(B(x, r)\big)\leq \eta\left(\frac{\kappa s}{r}\right) b.$$
	This implies that
\beq\label{1-7-4} \diam f(B(x, s))\leq \eta\left(\frac{\kappa s}{r}\right) b.\eeq
	
	For the remaining case, that is, $s<r$, the fact $B(x, s)\subset B(x, r)$ leads to
	\beq\label{1-7-5} \diam f(B(x, s))\leq \diam f\big(B(x, r)\big)\leq b.\eeq

If $B(x, r)\setminus B(x, s)=\emptyset$, apparently,
	\beq \label{1-7-6}
	\diam f(B(x, s))=\diam f\big(B(x, r)\big)\geq a,
\eeq
and if $B(x, r)\setminus B(x, s)\not=\emptyset$, then the similar reasoning as in the proof of \eqref{1-7-4} ensures that
\beq \label{1-7-7}
\diam f(B(x, s))\geq \frac{a}{\eta\left(\frac{\kappa r}{s}\right)}.
\eeq

Now, we conclude from \eqref{1-7-2}$-$\eqref{1-7-7} that for all $x\in Z$,
	$$a_1\leq \diam f(B(x, s))\leq b_1,$$
where
$$a_1=\min\left\{a,\;\frac{a}{\eta\left(\frac{\kappa r}{s}\right)}\right\}\;\;\mbox{and}\;\;b_1=\max\left\{b,\;\eta\left(\frac{\kappa s}{r}\right)b\right\}.$$	
This shows that $f$ is $s$-uniformly bounded.
\end{proof}

\begin{proof}[{\bf Proof of Theorem \ref{thm-2}}]
	$(1)\Rightarrow(2)$.  Assume that $f$ is locally $(\theta, 1/\theta, r)$-biH\"{o}lder continuous with $\theta\geq 1$ and $0<r<2\diam Z$. Then there is $C\geq 1$ such that for any $x\in Z$ and any $y\in B(x, r)$,
	\begin{equation}\label{lem-10-11}
		C^{-1}d_Z(x, y)^{\theta} \leq d_W(f(x), f(y)) \leq Cd_Z(x, y)^{1/\theta},
	\end{equation}
	which leads to
	$$
	\diam f\big(B(x, r)\big)\leq 2Cr^{1/\theta}.
	$$

Moreover, it follows from Lemma \ref{1-8-4} that there is $z\in Z$ such that
$$\frac{r}{\mu}\leq d_Z(x, z)<r,$$ where $\mu=\max\{8, \kappa\}$.
Then \eqref{lem-10-11} leads to
	$$\diam f\big(B(x, r)\big)\geq d_W(f(x), f(z))\geq 	\frac{r^\theta}{\mu^\theta C}.$$
	
	These show that $f$ is $r$-uniformly bounded.
	
	$(2)\Rightarrow(1)$.
	Assume that $f$ is $r$-uniformly bounded with $0<r<2\diam Z$. This assumption implies that there are two constants $a>0$ and $b>0$ such that for any $\xi\in Z$,
\begin{equation}\label{l1-8-1}
	a\leq \diam f(B(\xi, r))\leq b.
	\end{equation}

Let $x\in Z$. By Lemma \ref{1-8-4}, we see that there is $\zeta\in Z$ such that
\begin{equation}\label{1-8-7}
	\frac{r}{\mu}\leq d_Z(x, \zeta)<r,
\end{equation}
where $\mu=\max\{8, \kappa\}$.
We assert that
	\begin{equation}\label{lemma-bdd}
		\frac{a}{3\lambda {\mu}^\theta} \leq d_W(f(x), f(\zeta))\leq b.
	\end{equation}

	The right-side inequality of \eqref{lemma-bdd} easily follows from \eqref{l1-8-1}. For the proof of the left-side inequality, let $\zeta_1\in B(x, r)$ be such that
	$$d_W(f(x), f(\zeta_1))\geq \frac {1}{3}\diam f\big(B(x, r)\big),$$
and then, it follows from \eqref{l1-8-1} that
$$d_W(f(x), f(\zeta_1))\geq  \frac a 3.$$	

Since
	$$\frac{d_Z(x, \zeta_1)}{d_Z(x, \zeta)}\leq \mu,$$
	we know from the assumption of $f$ being $(\theta, \lambda)$-power quasisymmetric with $\theta\geq 1$ and $\lambda\geq 1$ that
	$$\frac{d_W(f(x), f(\zeta_1))}{d_W(f(x), f(\zeta))} \leq \lambda {\mu}^{\theta}.$$
	Hence
	$${d_W(f(x), f(\zeta))} \geq \frac{1}{\lambda {\mu}^\theta}d_W(f(x), f(\zeta_1)) \geq \frac{a}{3\lambda {\mu}^\theta},$$
	which is what we need. Thus the estimates in \eqref{lemma-bdd} are proved.

Let $y\in Z$ be such that $$d_Z(x, y)<r.$$
	
If $d_Z(x, y)\geq d_Z(x, \zeta)$, then $r/{\mu}\leq d_Z(x, y)<r$. It follows from \eqref{lemma-bdd} that
	\beq\label{1-8-5}
\frac{a}{3\lambda (r\mu)^\theta} d_Z(x, y)^{\theta} \leq \frac{a}{3\lambda {\mu}^\theta} 	\leq d_W(f(x), f(y))\leq   b\leq \frac{b\mu^{\frac{1}{\theta}}}{r^{\frac{1}{\theta}}} d_Z(x, y)^{\frac{1}{\theta}}.
\eeq

If $d_Z(x, y)<d_Z(x, \zeta)$, it follows from the assumption of $f$ being $(\theta, \lambda)$-power quasisymmetric that
	\begin{equation*}
		\lambda^{-1} \left(\frac{d_Z(x, y)}{d_Z(x, \zeta)}\right)^{\theta}\leq \frac{d_W(f(x), f(y))}{d_W(f(x), f(\zeta))}\leq \lambda\left(\frac{d_Z(x, y)}{d_Z(x, \zeta)}\right)^{1/\theta},
	\end{equation*}
and then, we deduce from \eqref{1-8-7} and \eqref{lemma-bdd} that
\beq\label{1-8-6}
\frac{a}{3\lambda^2 {(r\mu)}^\theta } d_Z(x, y)^{\theta}\leq d_W(f(x), f(y))\leq \frac{\lambda b\mu^{\frac{1}{\theta}}}{r^{\frac{1}{\theta}}}d_Z(x, y)^{\frac{1}{\theta}}.
\eeq

Now, we conclude from \eqref{1-8-5} and \eqref{1-8-6} that
$f$ is locally $(\theta, 1/\theta, r)$-biH\"{o}lder continuous, and hence, the theorem is proved.
\end{proof}

\section{Examples}\label{sec-5}

As an application of Theorem \ref{thm-2}, in this section, we construct two examples. The first example gives a quasisymmetric mapping between unbounded uniformly perfect spaces, which is not locally biH\"older continuous.
In the second example, we construct a locally biH\"older continuous mapping between unbounded Alhfors regular spaces, which is not biH\"{o}lder continuous. This example, together with Remark \ref{sec5}$(ii)$ below, also illustrates that Theorem \ref{thm-1} is more general than \cite[Proposition 7.2]{BBGS}.

\begin{example}\label{ex-add}
	Let $f$ be the radial stretching $f(x)=|x|x$ of $(\real^2, |\cdot|)$,
where $|\cdot|$ denotes the usual Euclidean metric. Then $f$ is a power quasisymmetric mapping but not locally biH\"older continuous.
\end{example}

\begin{proof}
It follows from \cite[p. 49]{V1} or \cite[p. 309]{HKM} that $f$ is a quasiconformal mapping. It is a fundamental fact that quasiconformal self-mappings of Euclidean spaces with dimension at least two are quasisymmetric,  see for example Gehring \cite{G1} or Heinonen-Koskela  \cite{HK}. This fact implies that
$f$ is a quasisymmetric mapping. Then we know from Theorem A that $f$ is power quasisymmetric. Here, we refer interested readers to \cite{HK,V1} for the definitions of quasiconformal mappings.

Suppose on the contrary that $f$ is locally biH\"older continuous. By Theorem \ref{thm-2} and Proposition \ref{uniformly-bounded}, $f$ must be $1$-uniformly bounded. However, for any $(n, 0)\in \mathbb R^2$ with $n\in \mathbb N$, a direct computation gives that
$$\diam f\left(B\big((n, 0), 1\big)\right)\geq (n+1)^2-n^2= 2n+1,$$
which contradicts the uniform boundedness condition. We conclude that $f$ is not locally biH\"older continuous.
\end{proof}
	
\begin{example}\label{ex}
	Let $f$ be the following self-homeomorphism of $(\real^2, |\cdot|)$:
	\begin{equation*}
		f(x)=\left\{\begin{array}{cl}
        0,& x=0,\\
		\frac{x}{|x|}\cdot |x|^{\frac 12}, &0<|x|<1, \\
			x, &|x|\geq 1.
		\end{array}\right.
	\end{equation*}
Then the following statements hold.
\begin{enumerate}
	\item[$(1)$]
	$f$ is power quasisymmetric.
	
	\item[$(2)$]
	$f$  is locally biH\"older continuous.
	
	\item[$(3)$]
	 $f$ is not $(\theta_1, \theta_2)$-biH\"older continuous for any $\theta_1>0$ and $\theta_2>0$.
	
\end{enumerate}
\end{example}
\begin{proof}
		$(1)$ The statement $(1)$ in the example follows from a similar argument with the one in the proof of Example \ref{ex-add}.
		
		$(2)$ To show $f$ is locally biH\"older continuous, by  Theorem \ref{thm-2}, it suffices to show that $f$ is $r$-uniformly bounded for some $r>0$. Choose $r=2$. Then it is obvious from the definition of $f$ that for any $x\in \real^2$,
		$$2\leq \diam f(B(x, 2))\leq 6.$$
		This implies that $f$ is $2$-uniformly bounded, and hence, it is locally biH\"older continuous.
		
		$(3)$  Suppose on the contrary that $f$ is $(\theta_1, \theta_2)$-biH\"older continuous for some $\theta_1$ and $\theta_2$ with $\theta_1\geq \theta_2>0$. Then $f$ is a snowflake mapping since $(\mathbb R^2, |\cdot|)$ is unbounded. That is, there are constants $C\geq 1$ and $\theta>0$ such that for any pair of  $x$ and $y$,
		$$
		C^{-1}|x-y|^{\theta}\leq |f(y)-f(x)|\leq C|x-y|^{\theta}.
		$$
		
		However, for any $x$ with $|x|<1$,
		$$
		|f(x)-f(0)|=|x|^{\frac12},
		$$ which implies that $\theta = \frac12$;
		and for any $x$ with $|x|\geq1$,
		$$
		|f(x)-f(0)|=|x|,
		$$ which shows that $\theta = 1$. We conclude from this contradiction that $f$ is not $(\theta_1, \theta_2)$-biH\"older continuous for any $\theta_1>0$ and $\theta_2>0$.
	\end{proof}
	
	\begin{rem}\label{sec5}
$(i)$		Following the arguments in the proofs of statements $(2)$ and $(3)$ in the Example \ref{ex}, it is not difficult to show that the mapping $f$ in Example \ref{ex} is  locally $(1, \frac12, 2)$-biH\"older continuous. We omit the detailed computations here.
		
$(ii)$		 It is known that $\real^2$ is $2$-Ahlfors regular. Let $s>0$, $s^\prime>0$ and $p\geq 1$ be parameters such that $(s-2s^\prime)p\geq 2$. Then we see that all assumptions in Theorem \ref{thm-1} are satisfied. Therefore, we infer from Theorem \ref{thm-1} that
		$f$ induces a canonical bounded embedding $f_{\#}: B^{s}_{p, p}(\mathbb R^2)\rightarrow B^{ s^\prime}_{p, p}(\mathbb R^2)$ via composition.
	\end{rem}

\subsection*{Acknowledgments}

The second author (X. Wang) was partly supported by NNSF of China under the number 12071121, and the third author (Z. Wang) was partly supported by NNSF of China under the number 12101226.

\noindent
Address:

\noindent
MOE-LCSM, School of Mathematics and Statistics, Hunan Normal University, Changsha, Hunan 410081, P. R. China.

\bigskip

\noindent Manzi Huang,

\noindent{\it E-mail}:  \texttt{mzhuang@hunnu.edu.cn}

\medskip

\noindent Xiantao Wang,

\noindent{\it E-mail}:  \texttt{xtwang@hunnu.edu.cn}

\medskip

\noindent Zhuang Wang,

\noindent{\it E-mail}:  \texttt{zwang@hunnu.edu.cn}

\medskip

\noindent Zhihao Xu,

\noindent{\it E-mail}:  \texttt{734669860xzh@hunnu.edu.cn}


\vspace*{5mm}



\begin{thebibliography}{99}
	
	\bibitem{BB11} A.~Bj\"{o}rn and J.~Bj\"{o}rn, {\it Nonlinear potential theory on metric spaces},  EMS Tracts in Mathematics, 17. European Mathematical Society (EMS), Z\"{u}rich, 2011. xii+403 pp.
	
	\bibitem{BBS03} A.~Bj\"orn, J.~Bj\"orn and N.~Shanmugalingam, {\it The Dirichlet problem for p-harmonic functions on metric spaces}, J. Reine  Angew. Math. 556 (2003), 173-203.
	
	\bibitem{BBS} A.~Bj\"{o}rn, J.~Bj\"{o}rn and N.~Shanmugalingam, {\it Extension and trace results for doubling metric measure spaces and their hyperbolic fillings}, J. Math. Pures Appl. 159 (2022), 196-249.
	
	\bibitem{BBGS} A.~Bj\"{o}rn, J.~Bj\"{o}rn, T.~Gill and N.~Shanmugalingam, {\it Geometric analysis on Cantor sets and trees}, J. Reine Angew. Math. 725 (2017), 63-114.
	
%
	
	
	
	
	\bibitem{BoSi} G.~Bourdaud and W.~Sickel, {\it Changes of variable in Besov spaces}, Math. Nachr. 198 (1999), 19-39.
	
	\bibitem{BP03} M.~Bourdon and H.~Pajot, {\it Cohomologie $l_p$ et espaces de Besov}, J. Reine Angew. Math. 558 (2003), 85-108.
	
	\bibitem{G1} F. W. Gehring, {\it The definitions and exceptional sets for quasiconformal mappings}, Ann. Acad. Sci. Fenn. Math., 281 (1960), 1-28.
%
%
%
%
	
	
	
	
	\bibitem{GKS10} A.~Gogatishvili, P.~Koskela and N.~Shanmugalingam, {\it Interpolation properties of Besov spaces defined on metric spaces}, Math. Nachr. 283 (2010), 215-231.
	
		\bibitem{GKZ13} A.~Gogatishvili,  P.~Koskela and Y. Zhou, {\it Characterizations of Besov and Triebel-Lizorkin spaces on metric measure spaces}, Forum Math. 25 (2013), 787--819.
	
\bibitem{H96} P.~Haj\l asz, {\it Sobolev spaces on an arbitrary metric space},  Potential Anal. 5 (1996), 403--415.
	

	\bibitem{H03} P.~Haj\l asz, {\it Sobolev space on metric-measure spaces, in Heat kernels and analysis on manifolds, graphs and metric spaces (Paris 2002),} Contemp. Math. 338, American Mathematical Society, Providence (2003), 173-218.
	
	\bibitem{HP} P.~Haj\l asz and P. Koskela, {\it Sobolev met Poincar\'e}, Mem. Amer. Math. Soc. (2000), no. 688, x+101 pp.
	
	
	
	
	\bibitem{H} J.~Heinonen, {\it Lectures on analysis on metric spaces},  Universitext. Springer-Verlag, New York, 2001. x+140 pp.
	
	\bibitem{HKM92} J.~Heinonen, T.~Kilpel\"{a}inen and O. Martio, {\it Harmonic morphisms in nonlinear potential theory,}	Nagoya Math. J. 125 (1992), 115-140.
	
		\bibitem{HKM} J.~Heinonen, T.~Kilpel\"{a}inen and O. Martio, {\it Nonlinear potential theory of degenerate elliptic equations},  Dover Publications, Inc., Mineola, NY, 2006.

	\bibitem{HK} J.~Heinonen and P.~Koskela, {\it Quasiconformal maps in metric spaces with controlled geometry}, Acta Math. 181 (1998), 1-61.
	
	\bibitem{HKST} J.~Heinonen, P.~Koskela, N.~Shanmugalingam and J.~Tyson, {\it Sobolev Spaces on Metric Measure Spaces: An Approach Based on Upper Gradients}. Cambridge: Cambridge University Press, 2015.
	
	
	
	\bibitem{HenK} S.~Hencl and L.~Kleprl\'ik, {\it Composition of $q$-quasiconformal mappings and functions in Orlicz-Sobolev spaces}, Illinois J. Math. 56 (2012), 931-955.
	
	
	
	
%
	
%
%
	
	
	
	\bibitem{KXZZ} P.~Koskela,  J. Xiao, Y.~Zhang  and Y. Zhou, {\it A quasiconformal composition problem for the $Q$-spaces}, J. Eur. Math. Soc. (JEMS) 19 (2017), 1159-1187.
	
	\bibitem{KYZ10} P. Koskela, D. Yang and Y. Zhou, {\it A characterization of Haj\l asz-Sobolev and Triebel-Lizorkin spaces via grand Littlewood-Paley functions}, J. Funct. Anal. 258 (2010), 2637-2661.
	
	
	\bibitem{KYZ} P.~Koskela, D.~Yang and Y.~Zhou, {\it Pointwise characterizations of Besov and Triebel-Lizorkin spaces and quasiconformal mappings}, Adv. Math. 226 (2011), 3579-3621.
	
	
	
	
	
	
	\bibitem{Rie} H. M. Riemann, {\it Functions of bounded mean oscillation and quasiconformal mappings}, 	Comment. Math. Helv. 49 (1974), 260-276.
	
\bibitem{N00} N. Shanmugalingam, \emph{Newtonian spaces: An extension of Sobolev spaces to metric measure spaces}, { Rev. Mat. Iberoam.} 243-279 ,Vol. 16, 2000.

	
	
%
%
%
	
	\bibitem{TV} P.~Tukia and J.~V\"{a}is\"{a}l\"{a}, {\it Quasi-symmetric embeddings of metric spaces}, Ann. Acad. Sci. Fenn. Ser. A I Math. 5 (1980), 97-114.
	
	\bibitem{V1} J.~V\"{a}is\"{a}l\"{a}, {\it Lectures on $n$-dimensional quasiconformal mappings}, Lecture Notes in Mathematics, Vol. 229. Springer-Verlag, Berlin-New York, 1971.
%
	
	\bibitem{Vo} S. K. Vodopyanov, {\it Mappings of homogeneous groups and embeddings of function spaces},	Sibirsk. Mat. Zh. 30 (1989), 25-41.
	
\end{thebibliography}
\end{document}